\documentclass[12pt]{amsart}
\usepackage{amsfonts,amssymb,latexsym,amsmath, amsxtra}
\usepackage[all]{xy}
\usepackage[dvips]{graphics}

\usepackage[OT2,T1]{fontenc}
\DeclareSymbolFont{cyrletters}{OT2}{wncyr}{m}{n}
\DeclareMathSymbol{\Sha}{\mathalpha}{cyrletters}{"58}

\theoremstyle{plain}
\newtheorem{theorem}{Theorem}[section]

\newtheorem{lemma}[theorem]{Lemma}
\newtheorem{proposition}[theorem]{Proposition}
\theoremstyle{definition}
\newtheorem{definition}[theorem]{Definition}
\theoremstyle{remark}
\newtheorem*{remark}{Remark}

\numberwithin{equation}{section}

\newcommand{\R}{\mathbb R}
\newcommand{\N}{\mathbb N}
\newcommand{\Z}{\mathbb Z}
\newcommand{\C}{\mathbb C}

\def\matr#1#2#3#4{\left(\begin{array}{cc}#1&#2\\#3&#4\end{array}\right)}

\def\cM{\mathcal M}
\def\H{\mathbb H}

\def\cM{\mathcal M}

\def\PP{\mathbb P}

\def\({\left(}
\def\){\right)}
\def\<{\left<}
\def\>{\right>}

\newcommand{\ol}[1]{\overline{{#1}}}
\newcommand{\wt}[1]{\widetilde{#1}}

\newcommand{\SL}{\text{SL}}

\newcommand{\abs}[1]{\left|#1\right|}

\def\cI{\mathcal{I}}
\def\sgn{\text{sgn}}
\def\eps{\epsilon}

\newcommand{\smatr}[4]{\(\begin{smallmatrix} #1 & #2 \\ #3 & #4\end{smallmatrix}\)}
\def\vareps{\varepsilon}

\begin{document}

\title[Linear Relations Among Poincar\'e Series]{Linear Relations Among Poincar\'e Series \\via Harmonic Weak Maass Forms}
\author{Robert C. Rhoades}
\address{Stanford University, Department of Mathematics, Bldg 380, Stanford, CA 94305}
\email{rhoades@math.stanford.edu}

\thanks{Research of the author was supported by an
NSF Mathematical Sciences Postdoctoral Fellowship.
Part of this work was done while supported by the chair in Analytic Number Theory at 
Ecole Polytechnique F\'ed\'erale de Lausanne.}

\date{\today}
\thispagestyle{empty} \vspace{.5cm}
\begin{abstract}
We discuss the problem of the vanishing of Poincar\'e series.   
This problem is known to be related to the existence of weakly holomorphic forms 
with prescribed principal part.  
The obstruction to the existence is related to the pseudomodularity of Ramanujan's mock theta functions. 
We embed the space of weakly holomorphic modular forms into the larger space of harmonic weak Maass forms.
From this perspective we discuss the linear relations between Poincar\'e series and the connection to Ramanujan's mock theta functions. 
\end{abstract}

\maketitle

\section{Introduction}
A very important class of modular forms is constructed via the method of averaging. 
Throughout let $z = x+iy$ be in the complex upper half plane.
For $m\in \N$,  $k
\in \frac{1}{2}\Z$ with $k\ge 2$ and $N\in \N$ such that $4\mid N$
when $k\in \frac{1}{2}\Z \setminus \Z$,
the classical Poincar\'{e} series at `$\infty$' for the group $\Gamma_0(N)$  are
defined by
\begin{equation}
P(m,k,N;z):= \sum_{\gamma \in \Gamma_\infty \backslash \Gamma_0(N)}
(j(\gamma, z))^{-2k} e(m\gamma z),\end{equation}
where $\Gamma_\infty = \{ \smatr{1}{t}{}{1}: t\in\Z\}$ is the stabilizer of the cusp $\infty$. 
  (see Chapter 3 of \cite{iwaniec}, for instance).   Here  $e(z) := e^{2\pi i
z}$ and $j(\gamma,z)$ is defined by
\begin{equation}\label{eqn:j}j(A, z):=
\begin{cases}\sqrt{cz+d} & k \in \Z
\\ \(\frac{c}{d}\) \varepsilon_d^{-1} \sqrt{cz+d} & k \in \frac{1}{2}\Z
\setminus \Z
\end{cases}\end{equation}
with $\(\frac{\cdot}{\cdot}\)$ is the Legendre symbol and $\vareps_d$,
for odd integers $d$, is given by \begin{equation}\label{eqn:epsilon}
\vareps_d:= \begin{cases}1 & d \equiv 1\pmod{4} \\ i & d\equiv
3\pmod{4}
\end{cases}.
\end{equation}

The set $\{P(m,k,N;z)\}_{m\ge 1}$ spans the
finite dimensional space of cusp forms on $\Gamma_0(N)$ of weight $k$, denoted $S_k(N)$. 
Moreover, for each $m$, 
the Petersson inner product, denoted  $\<\cdot, \cdot\>$, 
of $P(m,k,N;z)$ with any  $g\in S_k(N)$ is a constant times
the 
$m$th Fourier coefficient of $g$.  That is, 
\begin{equation}\label{eqn:petersson}
\<P(m,k,N;\cdot), g(\cdot)\> = \frac{\Gamma(k-1)}{(4\pi m)^{k-1}} \ol{c_g(m)},
\end{equation}
when $g(z) = \sum_{n\ge 1} c_g(n) q^n$ and $q:= e^{2\pi i z}$. 
Beyond this, little is known about such
Poincar\'e series. For example, since the space $S_k(N)$ is finite
dimensional, there exist many relations among the Poincar\'e series.

Establishing the vanishing or 
non-vanishing of a Poincar\'e series is a challenging problem with a long history going back at least as far
as 
Poincar\'e's memoir on Fuchsian groups
\cite{poincare} (see p. 249).  See also the English translation by Stillwell \cite{poincareBook} (see pp.
199 -- 207). 
Hejhal \cite{hejhal} discusses the problem of the vanishing of Poincar\'e series 
in relation to the holomorphic projection map from the upper half plane to a compact Riemann surface. 
Also, Iwaniec \cite{iwaniec} (see p. 54) asks for all the linear relations between
these series as well as for which Poincar\'e series vanish identically.

It turns out this problem
is related to the question of whether there are weakly holomorphic
modular forms with a given principal part.  
A \emph{weakly
holomorphic modular form} is any meromorphic modular form whose
poles are supported at the cusps.
The \emph{extended principal part} at infinity of a weakly
holomorphic modular form $f$ is the polynomial $P_{f, \infty} \in \C[q^{-1}]$
such that $f(z) -P_{f, \infty}(q^{-1}) = O\(e^{-\epsilon y} \)$ as $y\to \infty$.  The extended principal
part at other cusps is defined similarly.  Namely, if $\sigma$ is a cusp 
the extended principal part at $\sigma$ is the finite sum of  
terms in the Fourier expansion around $\sigma$ that do not have rapid decay toward $\sigma$.
Additionally, let the \emph{principal part at infinity} of $f$ 
be the extended principal part of $f$ minus the constant term.

For notational convenience, let $M_k^\infty(N)$ denote the space of weight $k$ 
weakly holomorphic modular forms on $\Gamma_0(N)$ with constant extended principal part at all 
cusps not equivalent to $\infty$. 

\begin{theorem}\label{thm:relation}
Let $k\in \frac{1}{2}\Z$ with $k \ge 2$, $N\in \N$ and $\cI$ be a
finite set of positive integers.  Let $\{\alpha_m\}_{m\in \cI}$ be a set of complex numbers. 
The following are equivalent
\begin{enumerate}
\item 
$$\sum_{m\in \cI} \alpha_m P(m,k,N;z) = 0$$
\item $$\sum_{m\in \cI} \frac{\alpha_m}{m^{k-1}} \ol{c_g(m)} = 0$$
for every $g(z) = \sum_{n=1}^\infty c_g(n) q^n \in S_k(N)$ 
\item There exists an element of $M_{2-k}^\infty (N)$ 
 with principal part at $\infty$ equal to
$$\sum_{m\in \cI} \frac{\alpha_m}{m^{k-1}} q^{-m}.$$
\end{enumerate}
\end{theorem}
\begin{remark}
A weakly holomorphic modular form of
non-positive weight is uniquely determined by the collection of its extended
principal parts at all cusps.  
\end{remark}

\begin{remark} 
We state this Theorem \ref{thm:relation} in terms of
relations between the Poincar\'e series at `$\infty$' and Fourier expansions at $\infty$.  
To handle
more general combinations of Poincar\'e series or Fourier coefficients
one needs to take a linear combination of
weakly holomorphic modular forms with principal part specified.
\end{remark}

This result has arose in many different lines of study.  For instance,  
hints of this appear  in the calculation of
exact formulas for the Fourier coefficients of weakly holomorphic modular forms.  
In that context the result may have first been 
discovered by Peterson in 1955 \cite{petersson}.  
As a result, it is sometimes referred to as the ``Peterson Principal Parts Condition''.  See Lehner's book 
\cite{lehnerBook} and the discussion starting on page 36 for more on this history. 
This result fits into the framework of Eichler cohomology. 
For such results, see the works of 
Knopp and Knopp-Mawi  \cite{knopp1, KM}.

The equivalence of parts (2) and (3) appears in the work of Siegel \cite{siegel}. 
Finally, Theorem \ref{thm:relation}  may also be constructed out of
Fay's study of the 
resolvent kernel for automorphic forms \cite{fay}.  

More recently this result has played a role in the study of Borcherds products 
and the analytic theory of Serre-Duality \cite{borcherds}.
Borcherds products give an explicit construction of a meromorphic modular form 
with prescribed divisor. 
Roughly, Serre duality for modular forms says that the only obstructions to finding a weakly holomorphic 
modular form of weight $2-k$ with 
given singularities are given by modular forms of weight $k$.   Therefore, Theorem \ref{thm:relation}
 is natural in that context.  

\section{Main Result}\label{sec:main}
We will give a reformulations and extension of Theorem \ref{thm:relation}. 
We begin by embedding the space of weakly holomorphic modular forms into
the space of harmonic weak Maass forms, 
a certain non-holomorphic generalizations of classical holomorphic modular forms. 
Define the weight $k$
hyperbolic Laplacian $\Delta_k$ by
\begin{equation}\label{deflap}
\Delta_k := -y^2\left( \frac{\partial^2}{\partial x^2}+
\frac{\partial^2}{\partial y^2}\right) + iky\left(
\frac{\partial}{\partial x}+i \frac{\partial}{\partial y}\right).
\end{equation} 

\begin{definition}\label{def:harmonic1}
A \emph{harmonic weak Maass form} of weight $k$ on
$$\Gamma = \Gamma_0(N)\subseteq
\begin{cases} \SL_2(\Z) & k\in \Z \\ \Gamma_0(4) & k\in \frac{1}{2} \Z\setminus
\Z\end{cases}$$  is any smooth function $f:\H\to
\C$ satisfying:
\begin{enumerate}
\item[(i)]
$$f \left(\frac{az+b}{cz+d}\right) = \begin{cases} (cz+d)^kf(z)& k \in \Z \\
\(\frac{c}{d}\)^{2k} \vareps_d^{-2k} (cz+d)^{k} f(z)& k \in
\frac{1}{2}\Z \setminus \Z\end{cases}$$
for all $\smatr{a}{b}{c}{d}
\in \Gamma$;
\item[(ii)] $\Delta_k f =0 $;
\item[(iii)]
There is a polynomial $G_{f,\infty}(z)=\sum_{n\leq 0} c_f^+(n)q^n \in
\C[q^{-1}]$ such that $f(z)-G_{f, \infty}(z) = O(e^{-\eps y})$ as $y\to\infty$
for some $\eps>0$.   Furthermore, analogous conditions hold at 
each of the other cusps $\sigma$ with $G_{f,\sigma}$ a constant. 
\end{enumerate}
\end{definition}
As before we refer to $G_{f, \infty}$ as the \emph{extended principal part} at infinity 
of the harmonic weak Maass form $f$.
Note, that one may 
generalize the growth condition $(iii)$ in various ways and it is often useful to do so.  See 
the recent work of Duke,  Imamoglu, and T\'oth
\cite{DIT} for an instance where there condition of (iii) is relaxed to allow linear exponential growth 
even after the subtraction of the extended principal part. 

Denote the space of harmonic weak Maass forms of weight $2-k$ for $\Gamma_0(N)$ by $H_{2-k}(N)$.  
By definition 
$$M_{2-k}^\infty(N) \subseteq H_{2-k}(N).$$
Extending the space of weakly holomorphic modular forms to that of harmonic weak Maass forms allows 
us to explicitly construct an automorphic form with any principal part.   The price we pay 
is that the forms may not be holomorphic.  

The differential operator
\begin{equation}
\xi_w:=2i y^w\cdot\overline{\frac{\partial}{\partial \overline{z}}},
\end{equation}
plays a central role in the study of such forms (see \cite{BF} for
example). Proposition 3.2 of \cite{BF} gives
\begin{equation}\label{eqn:ximap}
\xi_{2-k}: H_{2-k}(N)\longrightarrow S_k(N).
\end{equation}


This leads to the following result.

\begin{theorem}\label{thm:main} Let $k \ge2$ be in $\frac{1}{2}\Z$ and $N\in \N$. 
\begin{enumerate}
\item There exists a one-to-one correspondence between polynomials with vanishing constant term  
$F\(q^{-1}\) \in \C[q^{-1}]$ and harmonic weak Maass forms 
$\cM \in \H_{2-k}(N)$.

\item Given $F\(q^{-1}\) = \sum_{m=1}^d \frac{\alpha_m}{ m^{k-1}} \in \C[q^{-1}]$ and 
$\cM_F \in H_{2-k}(N)$ with principal part $F(q^{-1})$.  $\cM_F$ 
is a weakly holomorphic form if and only if $$\sum_{m=1}^d \ol{\alpha_m} P(m,k,N;z) = 0.$$
\item  With $F$ and $\cM_F$ as in the previous part,   
$\cM_F$ is weakly holomorphic if and only if $\xi_{2-k}\(\cM_F\) = 0$. 
\end{enumerate}
  
\end{theorem}

The development of the theory of harmonic weak Maass forms has been motivated by Borcherds products 
and Ramanujan's mock theta functions.  
In Section \ref{sec:ramanujan} we see that to establish the modularity of Ramanujan's mock theta function $f(q)$
one must add it a certain non-holomorphic function.
This results in a harmonic weak Maass form 
$\cM_f$.  Just as the non-holomorphicity of a form with arbitrary principal part is controlled by a nonzero 
cusp form,  the non-holomorphicity of the form $\cM_f$ completing Ramanujan's mock theta function $f$
is controlled by a nontrivial cusp form.

In the next section we develop a bit of the theory of harmonic weak Maass forms and 
give a proof of Theorem \ref{thm:main}.
 
\section{A Proof of Theorem \ref{thm:main}}\label{sec:proofs}

In this section we will give a proof of Theorem \ref{thm:main}. We begin by showing that every 
harmonic weak Maass form naturally decomposes into two parts.   
The interplay between these two parts, the 
non-holomorphic part and holomorphic parts, has played a significant role in 
the arithmetic of the Fourier coefficients of harmonic weak Maass forms, 
see for example the work of Bruinier, Ono and the author \cite{BOR}. 
We exploit this relationship to prove Theorem \ref{thm:main}.

It follows from \emph{(ii)} and \emph{(iii)} of Definition \ref{def:harmonic1} that every weight $2-w$
harmonic weak Maass form $\cM(z)$ has a Fourier expansion of the form
\begin{equation}\label{eqn:MaassFourier}
\cM(z)=\sum_{n\gg -\infty} c_\cM^+(n) q^n + \sum_{n<0} c_\cM^-(n)
\Gamma(w-1,4\pi |n|y) q^n,
\end{equation}
where $\Gamma(a,x)$ is the incomplete Gamma-function (see Section 3 of \cite{BF}). 
This suggests that $\cM(z)$ naturally decomposes
into two summands
\begin{eqnarray}
\label{eqn:fourierhh} &\cM^{+}(z):=\sum_{n\gg -\infty} c_\cM^+(n) q^n,\\
\label{eqn:fouriernh} &\cM^{-}(z):=\sum_{n< 0}
c_\cM^-(n)\Gamma(w-1,4\pi|n|y)q^n.
\end{eqnarray}
The function $\cM^+$ is referred to as the holomorphic part of $\cM$ and the function $\cM^-$ is referred to as the 
non-holomorphic part of $\cM$. 

It is not difficult to make \eqref{eqn:ximap} more
precise using Fourier expansions.
In the notation of \eqref{eqn:MaassFourier}
a straightforward calculation (see Section 4 of \cite{BOPNAS}, for example) shows that
$\xi_{2-k}(\cM)$ has the Fourier expansion
\begin{equation}\label{eqn:xiFourier}
\xi_{2-k}(\cM)=-(4\pi)^{k-1}\sum_{n\ge 1
}\overline{c_\cM^{-}(-n)}n^{k-1}q^n.
\end{equation}

Following Bruinier and Funke \cite{BF}, we define a pairing between the spaces $S_k(N)$
and
$H_{2-k}(N)$. For $g\in M_{k}(N)$ and $\cM \in H_{2-k}(N)$ we define
\begin{equation} \{g, \cM\}_k := (g, \xi_{2-k}(\cM))_{k} =
\int_{\Gamma_0(N) \backslash \H }g(z)\cdot \ol{\xi_{2-k}(f)(z)}
y^{k} \frac{dx \ dy}{y^2}.
\end{equation}
The pairing of $g$ and $\cM$ can be
explicitly evaluated in terms of the Fourier coefficients of $g$ and
the extended principal part of $\cM$. 
\begin{lemma}[Proposition 3.5 of \cite{BF}]\label{lem:BF3.5}
If $g(z) = \sum_{n>0} c_g(n) q^n\in S_k(N)$ and $\cM\in H_{2-k}(N)$ with Fourier expansion as in 
\eqref{eqn:MaassFourier}, then 
$$\{ g, \cM\}_k = \sum_{n>0}  c_g(n) c_\cM^+(-n).$$
\end{lemma}
Theorem 1.1 of \cite{BF}
gives that the pairing between $S_k(N)$ and
$H_{2-k}(N)/M_{2-k}^!(N)$ is non-degenerate. 
We recast this result into a form which is convenient for our use. 
\begin{lemma}\label{lem:nondegenerate}
Let $k \in \frac{1}{2}\Z$ and $k\ge 2$ and $\cM \in H_{2-k}(N)$.
 If
$\cM$ has the property that its
principal part at each cusp is
trivial, then $\xi_{2-k}(\cM)= 0$.
\end{lemma}
\begin{proof} Suppose that $g:=\xi_{2-k}(\cM) \ne 0$.
From  the definition of $\{\cdot, \cdot\}_k$ in terms of the Petersson inner product $\{\xi_{2-k}(\cM), g\}_k = \(g, g\)_k \ne
0$.  Applying Lemma \ref{lem:BF3.5} we see that  the principal part at infinity is non-trivial.
\end{proof}

\subsection{Maass-Poincar\'e Series}\label{sec:poincare}  In this section we will
describe two families of Poincar\'e series.  The first family is a family of Maass- Poincar\'e 
series.  These series allow us to construct an automorphic form with arbitrary principal part. 
The second family of series is the classical family of cuspidal Poincar\'e series. 

As usual, for
$A=\smatr{a}{b}{c}{d} \in \SL_2(\Z)$ and $f:\H\to \C$, we
let
\begin{equation}\label{slash}
(f\mid_k A )(z):= j(A,z)^{-2k} f(A z).
\end{equation}
 Let $m$ be an integer,
and let $\varphi_m:\R^{+}\to \C$ be a function which satisfies
$\varphi_m(y)=O(y^\alpha)$, as $y\to 0$, for some $\alpha\in \R$.
Define
\begin{equation}
\varphi^{*}_m(z):=\varphi_m(y)e(mx).
\end{equation}
Such functions are fixed by the translations $\Gamma_\infty:=\{ \pm
\matr{1}{n}{0}{1}\ :\ n\in \Z\}$.
Given  this data, for integers $N\geq 1$, we define the generic
Poincar\'e series
\begin{align}
\PP(m,k,\varphi_m,N;z):= \sum_{A\in\Gamma_\infty \backslash
\Gamma_0(N)}(\varphi^{*}_m \mid_k A)(z).
\end{align}

In this notation the classical family is given by
$$P(m,k,N;z)=\PP(m,k,e(imy),N;z).$$
We define a second family of Poincar\'e series, the Maass-Poincar\'e
series (see Section 1.3 of \cite{bruinier} or \cite{fay}). Let $M_{\nu,\,\mu}(z)$ be the
usual $M$-Whittaker function. For complex $s$, let
\begin{equation}\cM_s(y):= |y|^{-\frac{k}{2}}
M_{\frac{k}{2}\sgn(y),\,s-\frac{1}{2}}(|y|),
\end{equation}
and for $m\geq 1$ let $\varphi_{-m}(z):=\cM_{1-\frac{k}{2}}(-4\pi m
y)$. We let
\begin{equation}
Q(-m,k,N;z):=\frac{1}{(k-1)!}\cdot \PP(-m,2-k,\varphi_{-m},N;z).
\end{equation}

The Fourier expansions of these series are given in terms of the
$I$-Bessel and $J$-Bessel functions, and the Kloosterman sums
\begin{equation}
K_k(m,n,c):= \begin{cases} \sum_{v(c)^{\times}}
e\left(\frac{m\overline v+nv}{c}\right) & k \in \Z
\\\sum_{v(c)^{\times}} \(\frac{c}{v}\)^{2k}\varepsilon_v^{2k}
e\left(\frac{m\overline v+nv}{c}\right) & k \in
\frac{1}{2}\Z\setminus \Z \end{cases}.
\end{equation}
Here $v$ runs through the primitive residue classes modulo $c$, and
$v\overline v\equiv 1\pmod{c}.$ We have the following proposition
(for example, see Section 1.3 of  \cite{bruinier}).

\begin{proposition}\label{prop:QqExp}
If $k \ge 2$ with $k\in \frac{1}{2}\Z$ and $m, N\geq 1$, then 
$$Q(-m,k,N;z)= Q^{+}(-m,k,N;z) + Q^{-}(-m,k,N;z) \in H_{2-k}(N),$$
where
\begin{displaymath}
Q^{-}(-m,k,N;z) =-\frac{\Gamma(k-1,4\pi m y)}{(k-2)!}q^{-m}+
\sum_{n<0}b(-m,k,N;n) \cdot\Gamma(k-1,4\pi |n|y) q^n,
\end{displaymath}
with
\begin{displaymath}
b(-m,k,N;n) =  - \frac{2\pi i^{k}}{(k-2)!} \cdot
\left | \frac{m}{n}\right|^{\frac{k-1}{2}} \sum_{\substack{c>0\\
c\equiv 0\pmod{N}}} \frac{K_{2-k}(-m,n,c)}{c}\cdot J_{k-1}\left(
\frac{4\pi \sqrt{ |mn|}}{c}\right),
\end{displaymath}
for negative integers $n$ and where 
\begin{displaymath}
Q^{+}(-m, k, N;z)=q^{-m}+\sum_{n=0}^{\infty}b(-m,k,N;n)q^n,
\end{displaymath}
with
\begin{displaymath}
b(-m,k,N;0)=-\frac{(2 \pi
i)^{k}m^{k-1}}{(k-1)!}\cdot \sum_{\substack{c>0\\
c\equiv 0\pmod{N}}}\frac{K_{2-k}(-m,0,c)}{c^k},
\end{displaymath}
and
\begin{displaymath}
b(-m,k,N;n)=-2\pi i^{k} \cdot \sum_{\substack{c>0\\
c\equiv 0\pmod{N}}} \left(\frac{m}{n}\right)^{\frac{k-1}{2}}
\frac{K_{2-k}(-m,n,c)}{c}\cdot I_{k-1}\left(\frac{4\pi
\sqrt{|mn|}}{c}\right)
\end{displaymath}
for positive integers $n$.
Additionally, the principal part at all other cusps is zero.
\end{proposition}

More common is the Fourier expansion for the classical Poincar\'e
series.  It is given in the following proposition (see
Section 3.2 of \cite{iwaniec}, for instance).
\begin{proposition}\label{prop:poincClassical}
If $k\in \frac{1}{2} \Z$, $k\ge 2$, and $m, N\geq 1$  with $4\mid N$
when $k\in \frac{1}{2}\Z \setminus \Z$, then $P(m,k,N;z) =
\sum_{n\ge 1} a(m,k,N;n) q^n$, where
\begin{displaymath}
a(m,k,N;n)=2\pi i^{k}\left(\frac{n}{m}\right)^{\frac{k-1}{2}} \cdot
\sum_{\substack{c>0\\c\equiv 0\pmod{N}}} \frac{K_k(m,n, c)}{c}\cdot
J_{k-1} \left(\frac{4\pi \sqrt{mn}}{c}\right).
\end{displaymath}
\end{proposition}

Using Propositions \ref{prop:QqExp} and \ref{prop:poincClassical}
with \eqref{eqn:xiFourier} and the identity
$$\ol{K_{2-k}(-m,-n,c)} = K_k(m,n, c)$$ we see that
\begin{equation}\label{eqn:xiQpoincare}
\xi_{2-k}(Q(-m,k,N;z)) = \frac{(4\pi m)^{k-1}}{(k-2)!}
P(m,k,N;z).\end{equation}
See \cite{BKR} for a different proof of this result as well as additional properties of these Poincar\'e 
series and their Fourier coefficients. 

\begin{proof}[Proof of Theorem \ref{thm:main}]
We begin by showing that there is a one-to-one correspondence between principal parts and harmonic 
weak Maass forms.  First, let $F(q^{-1}) = \sum_{m =1}^d \beta_m q^{-m}$. Then by 
Proposition \ref{prop:QqExp} 
the harmonic weak 
Maass form 
$$\cM(z) := \sum_{m=1}^d \beta_m Q(-m, k, N;z)\in H_{2-k}(N)$$
has principal part at $\infty$ equal to $F(q^{-1})$.   
Additionally, if there is a second such harmonic weak Maass form, say $\cM'$, 
then the harmonic weak Maass form $\cM - \cM'$ will have vanishing principal part at all cusps (ie. the extended
principal parts are all constants). Applying Lemma \ref{lem:nondegenerate} we see that $\cM - \cM' = 0$.

To prove the second part of the theorem, 
assume that $\sum_{m\in \cI} \ol{\alpha_m} P(m,k,N;z) = 0$. Write
\begin{align*}f(z) = \sum_{m\in \cI} \frac{\alpha_m}{m^{k-1}} Q(-m,k,N;z)
= \sum_{n \gg -\infty} c_f^+(n)q^n + \sum_{n<0} c_f^-(n) \Gamma(k-1,
4\pi \abs{n} y) q^n.\end{align*} Applying \eqref{eqn:xiQpoincare} and \eqref{eqn:xiFourier}
we have $$\xi_{2-k}(f)(z) =
\frac{(4\pi)^{k-1}}{(k-2)!} \sum_{m\in \cI} \ol{\alpha_m} P(m,k,N;z) =
-(4\pi)^{k-1} \sum_{n\ge 1} \ol{c_f^-(-n)} n^{k-1}q^n.$$ By our
assumption on the sum over the Poincar\'e series we know that
$c_f^-(n) =0$ for all $n<0$.   But then $f\in M_{2-k}^\infty(N)$. So we
see that $$\sum_{m\in I} \frac{\alpha_m}{m^{k-1}} Q^+(-m,k,N;z)\in
M_{2-k}^\infty(N)$$ is the weakly holomorphic form that we desire.

Conversely, assume that such a weakly holomorphic modular form
exists. Call it $f$. 
From the expansion for the coefficients of $Q^+(-m,k,N;z)$ in
Proposition \ref{prop:QqExp} we may conclude that $\wt{f}(z) := -f(z) +
\sum_{m\in \cI} \frac{\alpha_m}{m^{k-1}} Q(-m,k,N;z) \in H_{2-k}(N)$
has trivial principal part at each cusp. By Lemma
\ref{lem:nondegenerate}, $0=\xi_{2-k}(\wt{f})= \sum_{m\in \cI}
\alpha_m P(m, k, N;\cdot).$

Finally, the last part of the theorem follows from \eqref{eqn:MaassFourier} and \eqref{eqn:xiFourier}.
\end{proof}

\section{Ramanujan's Mock Theta Functions and Harmonic Weak Maass Forms}\label{sec:ramanujan}
Ramanujan's mock theta functions proved themselves mysterious for more than 80 years.  
They are $q$-hypergeomtric series such as
\begin{equation*}
f(z):=1+\sum_{n=1}^{\infty}\frac{q^{n^2}}{(1+q)^2(1+q^2)^2\cdots (1+q^n)^2}.
\end{equation*}
Ramanujan's mock theta functions have ``nearly modular'' properties with respect to $z$, but fail to be fully modular.
Zwegers in his Ph.D. thesis \cite{Zw} explained that modularity  
is obtained if one adds  to the mock theta function the non-holomorphic function
$$f^-(z):= \pi^{-\frac{1}{2}} \sum_{n\equiv 1\pmod{6}} \sgn(n) \Gamma\(\frac{1}{2}, \frac{\pi n^2 y}{6}\)  q^{-\frac{n^2}{24} }.$$
The resulting function, $\cM_f(z) := q^{-1} f(24z)+f^{-} (24z)$, is a harmonic weak Maass form of weight $1/2$
on $\Gamma_0(144)$ with Nebentypus $\(\frac{12}{\cdot}\)$.

The principal part of the harmonic weak Maass form $\cM(z)$ is $q^{-1}$. 
Furthermore, the 
non-holomorphicity of $\cM(z)$ is dictated by the existence of a certain weight $3/2$ cusp form. 
Indeed, $f^-$ shares a close resemblance with the weight $3/2$ unary theta function
$$g(z):= \sum_{n \equiv 1 \pmod{6}} n q^{n^2/24}.
$$ In fact we have 
$$\xi_{\frac{1}{2} }\(\cM(z) \)= \xi_{\frac{1}{2} }\(f^-(24z) \)=  2g(24z).$$

The work of Bringmann and Ono \cite{bo1} for the construction of a Poincar\'e series that equal $\cM(z)$ and $g(z)$.

\section*{Acknowledgements}
to add?
%


\begin{thebibliography}{99}

\bibitem{borcherds} R. Borcherds, \emph{The Gross-Kohnen-Zagier theorem in higher dimensions.}
Duke Math. J. 97 (1999), no. 2, 219-233.

\bibitem{BKR} K. Bringmann, B. Kane, R. C. Rhoades, \emph{Duality and Differential Operators 
for Harmonic Maass Forms}, preprint. 

\bibitem{bo1} K. Bringmann and K. Ono,
\emph{The $f(q)$ mock theta function conjecture and partition
ranks}, Invent. Math. \textbf{165} (2006), 243--266.


\bibitem{BOPNAS} K. Bringmann and K. Ono,
\emph{Lifting cusp forms to Maass forms with an application to
partitions}, Proc. Natl. Acad. Sci., USA \textbf{104}, No. 10
(2007), 3725-3731.


\bibitem{bruinier} J. H. Bruinier, \emph{
Borcherds products on O(2, $l$) and Chern classes of Heegner divisors.} 
Lecture Notes in Mathematics, 1780. Springer-Verlag, Berlin, 2002. 

\bibitem{BF} J. H. Bruinier and J. Funke,
\emph{On two geometric theta lifts}, Duke Math. J. {\bf 125} (2004), 45-90.

\bibitem{BOR} J. H. Bruinier, K. Ono, R.C. Rhoades, \emph{Differential operators for harmonic weak
 Maass forms and the vanishing of Hecke eigenvalues}. Math. Ann. 342 (2008), no. 3, 673-693.


\bibitem{DIT} W. Duke, \"O. Imamoglu, and \'A. T\'oth, \emph{Cycle integrals of the $j$-function}, 
Ann. of Math. (2) 173 (2011), no. 2, 947-981.


\bibitem{DJ} W. Duke and P. Jenkins, \emph{On the zeros and coefficients of certain weakly 
holomorphic modular forms}. Pure Appl. Math. Q. 4 (2008), no. 4, 1327-1340.


\bibitem{fay} J.D. Fay, \emph{
Fourier coefficients of the resolvent for a Fuchsian group.}
J. Reine Angew. Math. {\bf 293/294} (1977), 143-203. 


\bibitem{hejhal} D. A. Hejhal,
\emph{Monodromy groups and Poincar\'e series}, 
Bull. Amer. Math. Soc. {\bf 84} (1978), no. 3, 339-376. 


\bibitem{iwaniec} H. Iwaniec, \emph{Topics in the classical theory
of automorphic forms}, Grad. Studies in Math. \textbf{17}, Amer.
Math. Soc., Providence, R.I., 1997.

\bibitem{knopp1} M. I. Knopp,
\emph{Some new results on the Eichler cohomology of automorphic forms}, 
Bull. Amer. Math. Soc. 80 (1974), 607-632. 

\bibitem{knopp} M. I. Knopp, \emph{Rademacher on $J(\tau)$, Poincar\'e
Series of Nonpositve Weights and the Eichler Cohomology}, Notices
Amer. Math. Soc.  {\bf 37} (1990),  no. 4, 385-393.



\bibitem{KM} M. I. Knopp and H. Mawi, \emph{Eichler Cohomology  Theorem for Automorphic 
Forms of Small Weight}, Proc. Amer. Math. Soc. {\bf 138} (2010), no. 2, 395-404. 






\bibitem{lehnerBook} J. Lehner, \emph{Discontinuous groups and automorphic functions.} Mathematical Survey, 
No. VIII American Mathematical Society, Providence, R.I. 1964






\bibitem{petersson} H. Petersson, \emph{\"Uber automorphe Formen mit Singularit\"aten im 
Diskontinuit\"atsgabiet}, Math. annalen {\bf 129}, 1955. 

\bibitem{poincare} H. Poincar\'e \emph{M\'emoire sur les fonctions fuchsiennes}, 
Acta Math. {\bf 1} (1882), 193-294.  

\bibitem{poincareBook} H. Poincar\'e,
\emph{Papers on Fuchsian functions.} 
Translated from the French and with an introduction by John Stillwell. Springer-Verlag, New York, 1985.





\bibitem{siegel} C. L. Siegel, \emph{Berechnung von Zetafunktionen an ganzzahligen Stellen}, 
Nachr. Akad. Wiss. G\"ottingen Math.-Phys. Kl. II {\bf 1969}  (1969), 87-102. 


\bibitem{Zw} S. Zwegers, \emph{Mock Theta Functions}, Thesis, Utecht, 2002.
\end{thebibliography}
\end{document}